\documentclass{lms}

\usepackage{amsfonts,amssymb,amsxtra,latexsym}
\usepackage{color,graphicx}
\usepackage[colorlinks=true]{hyperref}\hypersetup{urlcolor=blue, citecolor=blue, linkcolor=blue}

\makeatletter\@addtoreset{equation}{myequation}\makeatother
\usepackage{silence}
\newtheorem{theorem}{Theorem}[section]
\newtheorem{corollary}[theorem]{Corollary}
\newtheorem{lemma}[theorem]{Lemma}
\newtheorem{proposition}[theorem]{Proposition}

\newcommand{\beq}{\begin{equation}}
\newcommand{\eeq}{\end{equation}}
\newcommand{\be}[1]{\begin{equation}\label{#1}}
\newcommand{\ee}{\end{equation}}
\renewcommand{\(}{\left(}
\renewcommand{\)}{\right)}

\newcommand{\R}{\mathbb R}

\newcommand{\N}{{\mathbb N}}
\renewcommand{\S}{{\mathbb S}}
\newcommand{\nrm}[2]{\|{#1}\|_{\mathrm L^{#2}(\S^1)}}

\newcommand{\Pot}{\mathcal V}

\definecolor{darkgreen}{rgb}{0.2,0.7,0.1}

\newcounter{Hequation}

\makeatletter\g@addto@macro\equation{\stepcounter{Hequation}}\makeatother

\newcommand{\msc}[1]{\href{https://zbmath.org/classification/?q=cc:#1}{#1}}

\title[Monotonicity of the period]{Monotonicity of the period and positive\\ periodic solutions of a quasilinear equation}

\author{J.~Dolbeault, M.~Garc\'ia-Huidobro and R.~Man\'asevich}

\classno{\msc{34C25}, \msc{35J92} (primary), \msc{34L30}, \msc{34C23} (secondary).}

\extraline{J.D.~was partially supported by the Project EFI (ANR-17-CE40-0030) and \emph{Conviviality} (ANR-23-CE40-0003) of the French National Research Agency (ANR). M.G-H.~was supported by Fondecyt grant 1250522 from ANID-Chile, and R.M.~by Centro de Modelamiento Matem\'atico (CMM) BASAL fund FB210005 for center~of excellence from~ANID-Chile. The authors thank a referee for a question which is now addressed in Appendix~B.\\
\copyright\,\the\year\ by the authors. This paper may be reproduced, in its entirety, for non-commercial purposes.~\href{https://creativecommons.org/licenses/by/4.0/legalcode}{\scriptsize CC-BY~4.0}}

\begin{document}
\maketitle

\begin{abstract} We investigate the monotonicity of the minimal period of periodic solutions of quasilinear differential equations involving the $p$-Laplace operator. First, the monotonicity of the period is obtained as a function of a Hamiltonian energy in two cases. We extend to $p\ge2$ classical results due to Chow-Wang and Chicone for $p=2$. Then we consider a differential equation associated with a fundamental interpolation inequality in Sobolev spaces. In that case, we generalize monotonicity results by Miyamoto-Yagasaki and Benguria-Depassier-Loss to $p\ge2$.
\end{abstract}

\maketitle
\thispagestyle{empty}

\section{Introduction}\label{Sec:Intro}

In this paper we study monotonicity properties of the \emph{minimal period} of positive periodic solutions of
\be{potential}
\big(\phi_p(w')\big)'+\Pot'(w)=0\,,
\ee
where $p\ge 2$, $\phi_p(s)=|s|^{p-2}s$ and $w\mapsto(\phi_p(w'))'$ is the $p$-Laplace operator. The energy $E=|w'|^p/{p'}+\Pot(w)$ is conserved if $w$ solves~\eqref{potential}. Here $p'=p/(p-1)$ is the classical notation for the H\"older conjugate of $p$. Otherwise $'$ denotes the derivative of a function of one real variable. We are interested in the positive periodic solutions with energy less than $E_*:=\Pot(0)$ which are enclosed by an homoclinic orbit attached to $(w,w')=(0,0)$. We shall further assume that the nonnegative potential~$\Pot$ is such that these solutions are uniquely determined, up to translations, by the energy level~$E$, with minimal period $T(E)$. Our purpose is indeed to study under which conditions~$T$ is an increasing function of $E$ in the range $0\le E<E_*$. Furthermore we will consider the asymptotic behaviour of $T(E)$ as $E\to0_+$ and as $E\to(E_*)_-$. Surprisingly enough, the limit of $T(E)$ as $E\to0_+$ is different in the cases $p=2$ and $p>2$. We shall assume that $\Pot$ is a potential defined on $\R$ such that
\be{H1}\tag{H1}
\mbox{\parbox{10.4cm}{\emph{$\Pot$ is a $C^2$ function on $\R$ and there are $A$, $B\in\R$ with $0<A<B$ such that $\Pot(0)=\Pot(B)=E_*>0$, $\Pot'(0)=\Pot(A)=\Pot'(A)=0$, $\Pot''(A)>0$, and $0<\Pot(w)<E_*$ for all $w\in(0,A)\cup(A,B)$.}}}
\ee
See Fig.~\ref{Fig1}. The point $(w,w')=(A,0)$ is a stationary solution of~\eqref{potential} giving rise to a \emph{center} surrounded by closed periodic orbits, such that $0<w<B$, with minimal period $T(E)$. These periodic orbits are enclosed by the homoclinic orbit attached \hbox{to~$(0,0)$}.

\clearpage\begin{figure}[ht]\begin{center}
\includegraphics[width=\textwidth]{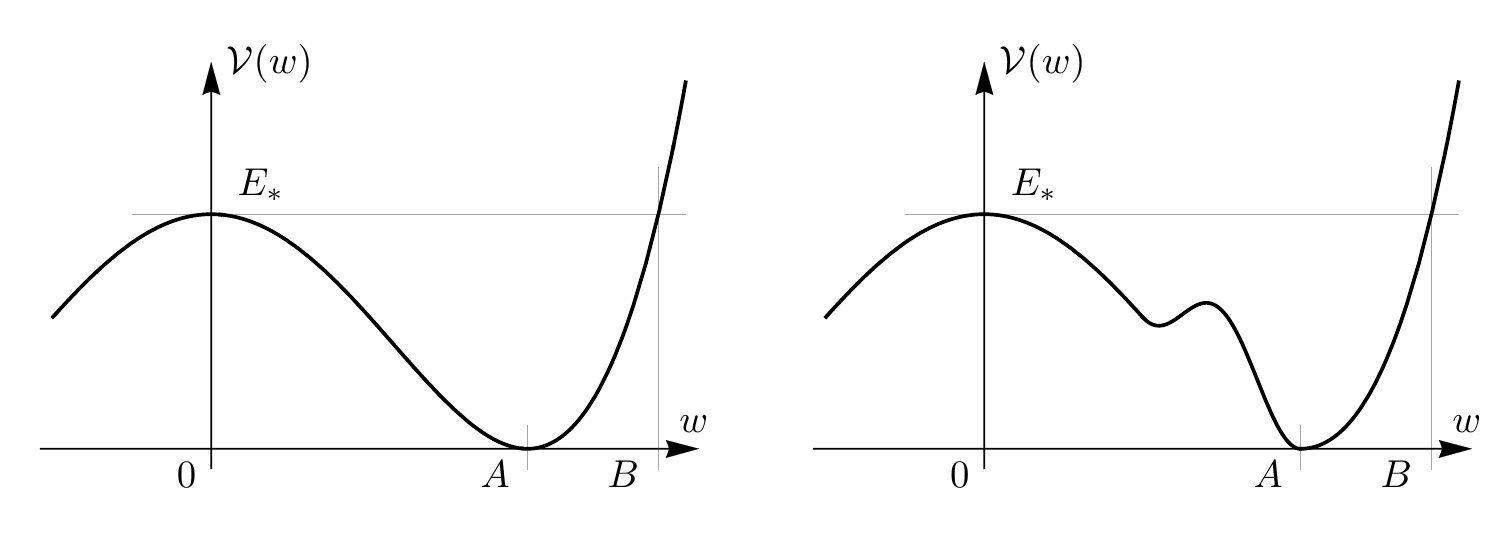}
\begin{caption}{\label{Fig1}\kern-2pt Two examples of $\Pot$ satisfying assumption~\eqref{H1}. Our monotonicity results require the stronger assumptions on $\Pot$ of Theorems~\ref{Thm:Main1} and~\ref{Thm:Main2} which, typically, hold in the left case but not in the right case.}\end{caption}
\end{center}\end{figure}

Our first result is based on an extension to $p>2$ of a result of Chow and Wang~\cite[Theorem~2.1]{Chow} for $p=2$ if $\Pot$ is smooth, positive, with a specific behaviour at $w=0$.
\begin{theorem}\label{Thm:Main1} Let $p>2$. Assume that $\Pot$ satisfies~\eqref{H1} and
\be{H2}\tag{H2}
(w-A)\,\Pot'(w)>0\quad\forall\,w\in(0,A)\cup(A,B)\,.
\ee
If $w\mapsto|\Pot'(w)|^2-p'\,\Pot(w)\,\Pot''(w)$ is positive, then $E\mapsto T(E)$ is increasing on $(0,E_*)$.\end{theorem}
\noindent Notice that $w\mapsto|\Pot'(w)|^2-p'\,\Pot(w)\,\Pot''(w)$ is a positive function if and only if the function $w\mapsto\Pot(w)\,|\Pot'(w)|^{-p'}$ is monotonically increasing. The right case in Fig.~\ref{Fig1} does not satisfy~\eqref{H2}.

 Our second result is an extension to $p>2$ of the monotonicity result~\cite[Theorem~A]{CH} for $p=2$, under \emph{Chicone's condition}, which is a growth condition as in Theorem~\ref{Thm:Main1}, but involves higher order derivatives of the potential $\Pot$.
\begin{theorem}\label{Thm:Main2} Let $p>2$. Assume that $\Pot$ is a $C^3$ function on $\R^+$ which satisfies~\eqref{H1}. If $\Pot/(\Pot')^2$ is a convex function, then $E\mapsto T(E)$ is increasing on $(0,E_*)$.\end{theorem}
A central motivation for this paper arises from the study of a minimization problem which is exposed in Appendix~\ref{App} (with additional references) and can be reduced to the study of all positive periodic solutions on $\R$ of
\be{ELS2}
\big(\phi_p(w')\big)'+\phi_q(w)-\phi_p(w)=0\,.
\ee
Equation~\eqref{ELS2} enters in the framework of~\eqref{potential} for $A=1$ and potential
\be{initpot}
\Pot(w)=\tfrac1q\,|w|^q-\tfrac1p\,|w|^p-\big(\tfrac1q-\tfrac1p\big)\,,
\ee
so that $E_*=1/p-1/q$. A positive periodic solution exists only if the energy level satisfies the condition $E<E_*$. Again we let $T(E)$ be the minimal period of such a solution. We cannot apply Theorems~\ref{Thm:Main1} and~\ref{Thm:Main2} (see Appendix~\ref{App:Comparison}) and we shall prove directly the following result, which is the main contribution of the paper.
\begin{theorem}\label{Thm:Main3} Let $p$ and $q$ be two real numbers such that $2<p<q$ and consider the positive periodic solutions of~\eqref{ELS2}. Then the map $E\mapsto T(E)$ is increasing on $(0,E_*)$ with $\lim_{E\to0_+}T(E)=0$ and $\lim_{E\to(E_*)_-}T(E)=+\infty$.\end{theorem}
Let us consider
\[
\mathcal H(u,v):=\Pot(u)+\tfrac1{p'}\,|v|^{p'}
\]
and $p'=p/(p-1)$, Equation~\eqref{potential} can be rewritten as the Hamiltonian system of equations
\[
u'=\frac{\partial\mathcal H}{\partial v}=\phi_{p'}(v)\quad\mbox{and}\quad v'=-\,\frac{\partial\mathcal H}{\partial u}=-\,\Pot'(u)
\]
with $w=u$ and $w'=\phi_{p'}(v)$. Although this Hamiltonian structure may look similar to those of~\cite[Theorem~1]{Schaaf}, we have a definitely different set of assumptions. In~\cite{Rothe_1993}, a very large set of Hamiltonian systems is considered but again our assumptions differ, for instance for the simple reason that the function $\phi_{p'}$ is not of class~$C^2$. Further references on the period function can be found in~\cite{Yagasaki-2013}. There are various other extensions of Chicone's result~\cite{CH}, see for instance~\cite{MR2143515}, but they do not cover the case of Theorem~\ref{Thm:Main3}. Notice that there is a computation in~\cite[Section~4]{MR2143515} which turns out to be equivalent to an argument used in the proof of our Theorem~\ref{Tchow} (see below in Section~\ref{section2}), although it is stated neither in that form nor as in~Theorem~\ref{Thm:Main1}.

The monotonicity of the minimal period as a function of the energy level is a question of interest by itself and particularly in the model case of the potential $\Pot$ given by~\eqref{initpot}, even in the case $p=2$. For this last case we quote from~\cite{BDL} that: ``It is somewhat surprising that, despite its ubiquity, the monotonicity of the period function for [this problem] in full generality was only established recently.'' In~\cite{Miyamoto_2013}, Miyamoto and Yagasaki indeed proved the monotonicity of the period function only for $p=2$ and for $q$ an integer. In~\cite{Yagasaki-2013}, Yagasaki generalized the result to all values of $q>2$. Both papers,~\cite{Miyamoto_2013,Yagasaki-2013}, rely on Chicone's criterion which is difficult to apply to non-integer values of $q$. Benguria, Depassier and Loss in~\cite{BDL} considered the positive periodic solutions of $w''+w^{q-1}-w=0$ (\emph{i.e.}, the case $p=2$ in our notation) and gave a simplified proof of the monotonicity of the period established in~\cite{Miyamoto_2013} using new changes of variables. Our contribution, in the case $p>2$, is a non-trivial extension of the $p=2$ case, it also relies on some of Chicone's ideas, but requires new techniques, especially in the case of Theorem~\ref{Thm:Main3}.

While the solutions of $-\Delta_pu+|u|^{q-2}\,u=\lambda\,|u|^{p-2}\,u$ have been studied for any $p>1$ in various cases including the one-dimensional case and periodic (but sign-changing) solutions, for instance in~\cite{MR965762,MR1151712}, most of the efforts have been devoted to the case $1<p<q$ and $\lambda>0$ and do not cover the case of Theorem~\ref{Thm:Main3}, which is in the class of $-\Delta_pu+|u|^{p-2}\,u=|u|^{q-2}\,u$ with $2<p<q$: in simpler words, positions of the terms $|u|^{p-2}\,u$ and $|u|^{q-2}\,u$ are exchanged. On the basis of numerical computations, it seems that properties of the solutions of~\eqref{ELS2} in the case $p\in(1,2)$ differ from the case $p>2$ as far as monotonicity properties of $E\mapsto T(E)$ are concerned and we will not study this issue here.

This paper is organized as follows. Section~\ref{section2} is devoted to the proofs of Theorems~\ref{Thm:Main1} and~\ref{Thm:Main2}. These results are classical in the limit case $p=2$ while we are not aware of such statements for $p>2$ in the existing literature. The result of Theorem~\ref{Thm:Main3} is by far more difficult. In Section~\ref{Section:Asymptotic} we start with problem~\eqref{potential} by making a change of variables and obtain an expression for the minimal period which goes along some of Chicone's ideas. We also prove some properties of the minimal period when the energy goes to zero and when it goes to the homoclinic level~$E_*$. In Section~\ref{section4} we establish a sufficient condition for the monotonicity of the minimal period which extends, in particular, the results of~\cite{BDL} for $p=2$ to the more general case of the one-dimensional $p$-Laplace operator with $p>2$. Our main result (Theorem~\ref{Thm:Main3}) is proved in Section~\ref{Section:application} using a new strategy of proof.

\section{A \texorpdfstring{$p$}p-Laplacian version of some classical results}\label{section2}

This section is devoted to the proof of Theorems~\ref{Thm:Main1} and~\ref{Thm:Main2}. We also provide a slightly more detailed statement of Theorem~\ref{Thm:Main1}.

We begin by extending~\cite[Theorem~2.1]{Chow} by Chow and Wang to the $p$-Laplacian case $p\ge2$. Equation~\eqref{potential} admits a first integral given by
\be{integral-0}
\tfrac1{p'}\,|w'|^p+\Pot(w)=E
\ee
for any energy level $E\in(0,E_*)$ and the minimal period is given in terms of the energy by
\be{periodp}
T(E)=\frac2{p'^{1/p}}\int_{w_1(E)}^{w_2(E)}\frac{dw}{\big(E-\Pot(w)\big)^{1/p}}
\ee
with $p'=p/(p-1)$, where $w_i(E)$, $i=1$, $2$, are two roots of $\Pot(w)=E$ such that
\[
0<w_1(E)<A<w_2(E)<B\quad\mbox{and}\quad \Pot(w)<E\quad\forall\,w\in\big(w_1(E),w_2(E)\big)\,.
\]
At this point, let us notice that the map $E\mapsto T(E)$ is a continuous function if we assume~\eqref{H2}, but that it is not the case if $\Pot$ admits another local minimum than $w=A$ in the interval $(A,B)$. See for instance Fig.~\ref{Fig1}. Let us define\[
\gamma(w,E):=p'\,\big(E-\Pot(w)\big)\,,\quad \mathcal R(w):=\Pot'(w)^2-p'\,\Pot(w)\,\Pot''(w)
\]
and notice that
\[
\frac{\partial\gamma}{\partial w}=-\,p'\,\Pot'(w)\quad\mbox{and} \quad \frac{\partial\gamma}{\partial E}=p'\,.
\]
The following result is a detailed version of Theorem~\ref{Thm:Main1}.
\begin{theorem}\label{Tchow} Let $p\ge2$ and consider Equation~\eqref{potential} where we assume that $\Pot$ satisfies~\eqref{H1}. With the above notations, for any $E\in(0,E_*)$, it holds that
\be{result1}
\frac{dT}{dE}(E)=\frac2{p'\,E}\int_{w_1(E)}^{w_2(E)}\frac{\mathcal R(w)}{\gamma(w,E)^{1/p}\,\Pot'(w)^2}\,dw
\ee
if the integral in the right hand side is finite. Thus if $\mathcal R$ is positive on $(0,A)\cup(A,B)$, then the minimal period is increasing.
\end{theorem}
{}From Assumption~\eqref{H1}, we know that $\Pot(0)=E_*>0$ and $\Pot'(0)=0$. As a consequence we have $\lim_{w\to0_+}\Pot(w)\,|\Pot'(w)|^{-p'}=+\infty$. On the other hand, by definition of $\mathcal R$, we have
\[
\(\frac \Pot{|\Pot'|^{p'}}\)'=\frac{\mathcal R\,\Pot'}{|\Pot'|^{{p'}+2}}\,.
\]
This is incompatible with $\mathcal R$ being a negative valued function in a neighbourhood of $w=0_+$. Alternatively, if we remove the assumption that $\Pot'(0)=0$, then it can make sense to assume that~$\mathcal R$ is a negative function on $(0,A)\cup(A,B)$. In this latter case, the minimal period is decreasing.
\begin{proof}[of Theorem~\ref{Tchow}] The proof relies on the same strategy as for~\cite[Theorem~2.1]{Chow}. We skip some details and emphasize only the changes needed to cover the case $p>2$. Let us set
\[
I(E):=\int_{w_1(E)}^{w_2(E)}\,\gamma(w,E)^{1/p'}\,dw\quad\mbox{and}\quad J(E):=\int_{w_1(E)}^{w_2(E)}\big(\gamma(w,E)-p'\,E\big)\,\gamma(w,E)^{1/p'}\,dw\,.
\]
By differentiating with respect to $E$, we obtain
\[
\frac{dI}{dE}(E)=\int_{w_1(E)}^{w_2(E)}\frac{dw}{\gamma(w,E)^{1/p}}=\frac12\,T(E)\,dw\quad\mbox{and}\quad \frac{dJ}{dE}(E)=\int_{w_1(E)}^{w_2(E)}\frac{\gamma(w,E)-p' E}{\gamma(w,E)^{1/p}}\,dw\,,
\]
which implies that
\[
\frac{dJ}{dE}(E)=I(E) -p'\,E\,\frac{dI}{dE}(E)\,.
\]
Differentiating once more with respect to $E$, we get
\be{Jsecond}
\frac{d^2\!J}{dE^2}(E)=(1-p')\,\frac{dI}{dE}(E)-p'\,E\,\frac{d^2\!I}{dE^2}(E)\,.
\ee

On the other hand, by integrating by parts in
\[
\int_{w_1}^{w_2}\,\gamma^{1+1/p'}\,\frac{\Pot'^2-\Pot\,\Pot''}{\Pot'^2}\,dw=\int_{w_1}^{w_2}\,\gamma^{1+1/p'}\(\frac \Pot{\Pot'}\)'\,dw=-\,\frac{1+p'}{p'}\int_{w_1}^{w_2}\,\gamma^{1/p'}\,\frac \Pot{\Pot'}\,\frac{\partial\gamma}{\partial w}\,dw\,,
\]
we obtain
\[
J(E)=-\,\frac{p'}{p'+1}\int_{w_1(E)}^{w_2(E)}\,\gamma(w,E)^{1+1/p'}\,\frac{\Pot'(w)^2-\Pot(w)\,\Pot''(w)}{\Pot'(w)^2}\,dw
\]
by definition of $J$ and $\gamma$. See~\cite{Chow} for further details in the case $p=2$. By differentiating twice this expression of $J(E)$ with respect to $E$, we obtain
\[
\frac{d^2\!J}{dE^2}(E)=-\,p'\int_{w_1(E)}^{w_2(E)}\frac{\Pot'(w)^2-\Pot(w)\,\Pot''(w)}{\gamma(w,E)^{1/p}\,\Pot'(w)^2}\,dw\,.
\]
Since $T(E)=2\frac{dI}{dE}(E)$, we learn from~\eqref{Jsecond} that
\begin{align*}
\frac{p'\,E}2\,\frac{dT}{dE}(E)=&\,p'\,E\,\frac{d^2\!I}{dE^2}(E)=(1-p')\,\frac{dI}{dE}(E)-\frac{d^2\!J}{dE^2}(E)\\
=&\,(1-p')\int_{w_1(E)}^{w_2(E)}\frac{dw}{\gamma(w,E)^{1/p}}+p'\int_{w_1(E)}^{w_2(E)}\frac{\Pot'(w)^2-\Pot(w)\,\Pot''(w)}{\gamma(w,E)^{1/p}\,\Pot'(w)^2}\,dw\\
=&\,\int_{w_1(E)}^{w_2(E)}\frac{\mathcal R(w)}{\gamma(w,E)^{1/p}\,\Pot'(w)^2}\,dw\,.
\end{align*}
This concludes the proof of~\eqref{result1}.
\end{proof}

\begin{proof}[of Theorem~\ref{Thm:Main2}]
Let us consider again Equation~\eqref{potential} with a potential $\Pot$ which satisfies~\eqref{H1}. We adapt the proof of~\cite[Theorem~A]{CH} to the case $p>2$. Let us consider the function 
\be{Eqn:H}
\mathsf H(w):=\frac{w-A}{|w-A|}\,\sqrt{\Pot(w)}
\ee
for any $w\in(0,A)\cup(A,B)$ and extend it by $\mathsf H(A)=0$ at $w=A$. With the notations of~\eqref{periodp}, we have $\mathsf H\big(w_1(E)\big)=-\,\sqrt E$, $\mathsf H\big(w_2(E)\big)=+\,\sqrt E$ and we can reparametrize the interval $\big(w_1(E),w_2(E)\big)$ with some $\theta\in(-\pi/2,\pi/2)$ such that
\[
\sqrt E\,\sin\theta=\mathsf H(w)\,.
\]
With this change of variables, the minimal period can be written as
\be{period:Chicone}
T(E)=2\,\frac{E^{\frac12-\frac1p}}{(p')^\frac1p}\int_{-\frac\pi2}^\frac\pi2\frac{(\cos\theta)^{1-2/p}}{\big(\mathsf H'\circ\mathsf H^{-1}\big)\big(\sqrt E\,\sin\theta\big)}\,d\theta\,.
\ee
Its derivative with respect to $E$ is given by
\[
\frac{dT}{dE}(E)=\big(\tfrac12-\tfrac1p\big)\,\frac{T(E)}E-(p'\,E)^{-\frac1p}\int_{-\frac\pi2}^\frac\pi2\frac{\mathsf H''(w)}{\mathsf H'(w)^3}\,(\cos\theta)^{1-2/p}\,\sin\theta\,d\theta
\]
where we use the short-hand notation $w=w(\theta)=\mathsf H^{-1}\big(\sqrt E\,\sin\theta\big)$. After an integration by parts and using $w'(\theta)=\sqrt E\,\cos\theta/\mathsf H'(w)$, this expression becomes
\[
\frac{dT}{dE}(E)=\big(\tfrac12-\tfrac1p\big)\,\frac{T(E)}E+\,\tfrac12\,(p')^\frac1{p'}\,E^{\frac12-\frac1p}\int_{-\frac\pi2}^\frac\pi2\frac{3\,\mathsf H''(w)^2-\mathsf H'(w)\,\mathsf H'''(w)}{\mathsf H'(w)^5}\,(\cos\theta)^{3-2/p}\,d\theta\,.
\]
At this point, we notice that
\[
3\,(\mathsf H'')^2-\mathsf H'\,\mathsf H'''=\frac{|\Pot'|^4}{8\,\Pot^2}\(\frac \Pot{|\Pot'|^2}\)''
\]
is positive if and only if $\Pot/(\Pot')^2$ is a convex function. This completes the proof of Theorem~\ref{Thm:Main2}.
\end{proof}

\section{Further properties of the minimal period}\label{Section:Properties}

Let us collect some properties of $T(E)$ for a potential $\Pot$ satisfying~\eqref{H1} and eventually~\eqref{H2}, which are not given in Theorems~\ref{Thm:Main1} and~\ref{Thm:Main2}.

\subsection{Asymptotic results}\label{Section:Asymptotic}

As in Section~\ref{section2}, recall that~\eqref{potential} has a first integral given by~\eqref{integral-0}, where $E\ge0$ is the energy level. In this short section, we shall assume that~\eqref{H1} holds, define
\be{omega}
\omega:=\sqrt{\Pot''(A)}>0
\ee
and make the additional hypothesis
\be{H3}\tag{H3}
\liminf_{w\to0_+}\frac{|\Pot'(w)|}{w^{p-1}}>0\,.
\ee
This assumption is satisfied in case of~\eqref{initpot} if $q>p>2$. In this case, we have $\omega=\sqrt{\Pot''(1)}=\sqrt{q-p}$, but the following result holds for a much larger class of potentials.
\begin{lemma}\label{asymp} Let $p>1$. If $\Pot$ is a potential such that~\eqref{H1} holds, then we have
\[
T(E)\sim\frac{2\,\sqrt{2\,\pi}\,\Gamma\big(1-\frac1p\big)}{(p')^\frac1p\,\omega\,\Gamma\big(\frac32-\frac1p\big)}\,E^{\frac12-\frac1p}\quad\mbox{as}\quad E\to0_+
\]
with $\omega$ defined by~\eqref{omega}. As a consequence, we obtain
\begin{align*}
\lim_{E\to0_+}T(E)=&\,0&\mbox{if}\quad p>2\,,\\[-8pt]
\lim_{E\to0_+}T(E)=&\,\frac{2\,\pi}\omega&\mbox{if}\quad p=2\,,\\
\lim_{E\to0_+}T(E)=&\,+\infty&\mbox{if}\quad p\in(1,2)\,.
\end{align*}
Additionally, if~\eqref{H3} holds, then for any $p>1$ we have $\displaystyle\lim_{E\to(E_*)_-}T(E)=+\infty$.
\end{lemma}
\begin{proof} In a neighbourhood of $w=A$, we can write $\Pot(w)\sim\frac12\,\omega^2\,(w-A)^2$, use~\eqref{periodp} and the change of variables $w=A+\sqrt{2\,E}\,y/\omega$ to obtain
\[
T(E)\sim\frac{2\,\sqrt2}{p'^{1/p}\,\omega}\,E^{\frac12-\frac1p}\int_{-1}^1\frac{dy}{\big(1-y^2\big)^{1/p}}\quad\mbox{as}\quad E\to0_+\,.
\]
We obtain the expression of the integral using the formulae~\cite[6.2.1 \& 6.2.2]{zbMATH03863589} for the Euler Beta function.

Now let us consider the limit as $E\to(E_*)_-$. We learn from~\eqref{H3} that
\[
E_*-\Pot(w)\ge\frac\ell p\,w^p
\]
for some $\ell>0$ if $w>0$ is taken small enough. We deduce from~\eqref{periodp} that $T(E)$ diverges as $E\to(E_*)_-$.
\end{proof}

\subsection{The monotonicity of the minimal period}\label{section4}

In this section, we develop a new sufficient condition for the monotonicity of the period to hold (see Proposition~\ref{Prop:K} below). This condition will be used in Section~\ref{Section:application} to prove Theorem~\ref{Thm:Main3}. Let us assume that the potential is defined by~\eqref{initpot}. Applying the same strategy of computation as in the proofs of Section~\ref{section2} involves very complicated expressions. For that reason, it is convenient to introduce a new change of variables as follows. Let us define
\be{D1}\tag{D1}
h(y):=\frac{y-A}{|y-A|}\,\sqrt{\Pot(y^{1/p})}\quad\forall\,y\in\big(0,A^p\big)\cup\big(A^p,B^p\big)
\ee
and extend it by $h\big(A^p\big)=0$ at $y=A^p$. Under Assumption~\eqref{H2}, $y_i(E)$, $i=1$, $2$, are the two roots in $(0,B)$ of $\Pot(y^{1/p})=E$. As in Theorem~\ref{Tchow}, $\Pot(y^{1/p})=E$ admits no other root in $(0,B)$ for any $E\in(0,E_*)$ and the map $E\mapsto T(E)$ is continuous. Assumption~\eqref{H1} implies the monotonicity of $\Pot$ neither on $(0,A)$ nor on $(A,B)$, while this monotonicity is granted under Assumption~\eqref{H2}. See Fig.~\ref{Fig1}. Also notice that
\[
 h'(y)>0\quad\forall\,y\in\big(y_1(E),A^p\big)\cup\big(A^p,y_2(E)\big)\,.
\]

We introduce the \emph{change of variables} $y\mapsto\theta\in(-\pi/2,\pi/2)$ such that
\be{ChangeOfVariables}
\sqrt E\,\sin\theta=h(y)\,.
\ee
Let us consider the function $h(y):=\mathsf H\big(w(y)\big)$ with $\mathsf H$ defined by~\eqref{Eqn:H}, where $w(y):=y^{1/p}-A$, so that $w'(y)=\frac1p\,y^{-1/p'}$. As in the proof of Theorem~\ref{Thm:Main2}, we use~\eqref{period:Chicone} to write
\be{period5}
T(E)=\mathsf c_p\,E^{\frac12-\frac1p}\int_{-\frac\pi2}^\frac\pi2\frac{(\cos\theta)^{1-2/p}}{y^{1/p'}\,h'(y)}\,d\theta\quad\mbox{with}\quad\mathsf c_p:=\frac2{{p\,(p')}^{1/p}}
\ee
with $y=y(E,\theta)$ defined by~\eqref{ChangeOfVariables}. Let us define
\[
J(E):=\int_{-\frac\pi2}^\frac\pi2\frac{(\cos\theta)^{1-2/p}}{y^{1/p'}\,h'(y)}\,d\theta
\]
and emphasize that $J$ is a function of $E$ as a consequence of the change of variables~\eqref{ChangeOfVariables}: $y=y(E,\theta)$ satisfies
\[
\frac{\partial y}{\partial E}=\frac{\sin\theta}{2\,\sqrt E\,h'(y)}\,.
\]
By differentiating $T(E)$ in~\eqref{period5} with respect to $E$ and using the relation~\eqref{ChangeOfVariables}, we find that
\[
\frac{T'(E)}{T(E)}=\frac{p-2}{2\,p}\,\frac1E+\frac{J'(E)}{J(E)}
\]
where
\[
J'(E)=-\,\frac1{2\,\sqrt E}\int_{-\frac\pi2}^\frac\pi2\,K(y)\,(\cos\theta)^{1-2/p}\,\sin\theta\,d\theta
\]
and
\be{fK}
K(y):=-\,\frac1{h'(y)}\,\frac d{dy}\(\frac1{y^{1/p'}\,h'(y)}\)=
\frac{y^2\,h''(y)+\frac1{p'}\,y\,h'(y)}{y^{2+1/p'}\,\big(h'(y)\big)^3}\,.
\ee
For $p>2$, $E\mapsto T(E)$ is increasing if $J'(E)>0$. Here is a sufficient condition on $h$, which is in fact an assumption on~$\Pot$.
\begin{proposition}\label{Prop:K} Assume that~\eqref{H1} and~\eqref{H2} hold. With the above notations, if the function~$K$ is decreasing on $[0,B^p]$, then $J'>0$ on $(0,E_*)$ and the minimal period $T(E)$ is a monotonically increasing function of $E$.
\end{proposition}
\begin{proof} With $y(E,\theta)$ defined by~\eqref{ChangeOfVariables}, the result is a consequence of
\[
J'(E)=-\,\frac1{2\,\sqrt E}\int_0^\frac\pi2\,\Big(K\big(y(E,\theta)\big)-K\big(y(E,-\,\theta)\big)\Big)\,(\cos\theta)^{1-2/p}\,\sin\theta\,d\theta
\]
and $y(E,-\,\theta)<y(E,\theta)$ if $\theta\in(0,\pi/2)$.
\end{proof}

In the next result, we give a sufficient condition on $h$ so that the monotonicity assumption on $K$ in Proposition~\ref{Prop:K} is satisfied and hence its conclusion holds.
\begin{corollary}\label{maincor} Assume that~\eqref{H1} and~\eqref{H2} hold. If $h$ and and $1/h'^2$ are convex functions, then the minimal period $T(E)$ is a monotonically increasing function of $E\in(0,E_*)$.
\end{corollary}
\begin{proof}
By convexity of $1/h'^2$, we have that
\[
0<\frac12\frac{d^2}{dy^2}\(\frac1{h'^2}\)=-\,\frac d{dy}\(\frac{h''}{h'^3}\)
\]
and $\frac{h''}{h'^3}$ is a decreasing function. Next, by~\eqref{fK}, we have
\[
K(y)=\frac1{y^{1/p'}}\,\frac{h''(y)}{\big(h'(y)\big)^3}+\frac1{p'}\,\frac1{y^{2-1/p}}\,\frac1{\big(h'(y)\big)^2}\quad\mbox{if}\quad y>0
\]
and observe that all the factors in the right hand side are positive decreasing functions, implying that $K$ is a decreasing function on $[A,B]$.
\end{proof}

\section{Proof of the main result}\label{Section:application}

In this section, we prove Theorem~\ref{Thm:Main3}, corresponding to the case of a potential $\Pot$ specifically given by~\eqref{initpot}, which satisfies Assumptions~\eqref{H1} and~\eqref{H2}, but is not covered by the other assumptions in Theorems~\ref{Thm:Main1} and~\ref{Thm:Main2}. For the sake of clarity, this is done by considering separately the cases $q=2\,p$, $q>2\,p$ and $p<q<2\,p$.

\subsection{Notations and strategy}\label{Section:notations}

It is convenient to define
\[
W(y):=y^m-m\,y+m-1\quad\forall\,y\in[0,\gamma_m]
\]
and
\be{D2}\tag{D2}
h(y):=\frac{y-1}{|y-1|}\,\sqrt{W(y)/q}\quad\forall\,y\in[0,\gamma_m]
\ee
where
\[
m:=\frac qp\quad\mbox{and}\quad\gamma_m:=m^\frac1{m-1}\,. 
\]
The link with the notations of Section~\ref{Sec:Intro} and the framework of Theorem~\ref{Thm:Main3} goes as follows: if $\Pot$ is defined by~\eqref{initpot}, then $W(y)=q\,\Pot\big(y^{1/p}\big)$, $q\,E_*=m-1$, $\gamma_m=B=(q/p)^{p/(q-p)}$ and $A=1$. Notice that Definitions~\eqref{D1} and~\eqref{D2} coincide. With these notations, we have
\[
0=W(1)<W(y)<W(0)=W(\gamma_m)=m-1\quad\forall\,y\in(0,1)\cup(1,\gamma_m)\,.
\]
The \emph{change of variables} $y\mapsto\theta\in(-\pi/2,\pi/2)$ defined by~\eqref{ChangeOfVariables} amounts to
\[
\sqrt E\,\sin\theta=h(y)\,.
\]
We consider the cases $m=2$, $m>2$ and $1<m<2$, respectively, in Sections~\ref{5.2},~\ref{5.3}, and~\ref{5.4}, corresponding to $q=2\,p$, $q>2\,p$ and $p<q<2\,p$.

\subsection{The case \texorpdfstring{$m=2$}{m=2}}\label{5.2}

As a special case, note that $W(y)=(y-1)^2$ and $h(y)=(y-1)/\sqrt q$ if $m=2$. In that case, the result of Theorem~\ref{Thm:Main3} is straightforward.
\begin{lemma}\label{Lem:m=2} If $m=2$, the minimal period $T(E)$ is a monotonically increasing function of~$E\in(0,E_*)$ with $E_*=\frac1{2\,p}$.\end{lemma}
\begin{proof} The function $K$ defined by~\eqref{fK} is explicitly given by $K(y)=\frac{q^2}{p'}\,y^{-1/p}$ hence monotonically decreasing and Proposition~\ref{Prop:K} applies.\end{proof}

\subsection{The case \texorpdfstring{$m>2$}{m>2}}\label{5.3}

We start with the following result.
\begin{lemma}\label{Lem:m>2} If $h$ is given by~\eqref{D2} and $m>2$, then $h$ and $(h')^{-2}$ are convex.
\end{lemma}
\begin{proof} Let $z=y^{m-1}$. For $0\le y\le\gamma_m$, we find that the expression
\[
4\,W^{3/2}\,\big(\sqrt W\,\big)''=2\,W\,W''-\big(W'\big)^2
\]
has the same sign as
\[
F(y):=-\,m^2+2\,m\,(m-1)^2\,y^{m-2}-2\,m^2\,(m-2)\,y^{m-1}+m\,(m-2)\,y^{2m-2}\,.
\]
$\bullet$ If $y\ge1$, then $y^{m-2}\ge1$, $-\,m^2+2\,m\,(m-1)^2\,y^{m-2}\ge m\,(m-2)\,(2\,m-1)$ and
\[
F(y)\ge m\,(m-2)\,(z-1)\,(z+1-2\,m)\ge0\,.
\]
$\bullet$ If $y\le1$, then $y^{m-2}\le1$, $y^{2m-2}\le1$, $m\,(m-2)\,y^{2m-2}-m^2\le-\,2\,m\,y^{2m-2}$ and
\[
-\,F(y)\ge2\,m\,(z-1)\(z+(m-1)^2\)\ge0\,.
\]
In both cases, we conclude using~\eqref{D2} that $h''$ has the same sign as $\frac{y-1}{|y-1|}\,F(y)$.

Using~\eqref{D2} again, we notice that the function $\big((h')^{-2}\big)''$ has the same sign as
\begin{multline*}
G(y):=2\,(m-1)\,(m-2)-m\,(2\,m-1)\,y\\
+2\,(m-1)\,(2\,m-1)\,y^{m-1}-2\,(m-2)\,(2\,m-1)\,y^m+(m-2)\,y^{2m-1}\,.
\end{multline*}
Since $G(1)=G'(1)=0$ and
\[
G''(y)=2\,(m-1)\,(m-2)\,(2\,m-1)\,W(y)\ge0\,,
\]
we conclude that $G\ge0$ and $\((h')^{-2}\)''\ge0$.
\end{proof}
Then, as a straightforward consequence of Lemmas~\ref{Prop:K} and~\ref{Lem:m>2}, we have
\begin{lemma}\label{Lem:m>22} If $m>2$, then the minimal period $T(E)$ is a monotonically increasing function of~$E\in(0,E_*)$.\end{lemma}
\begin{proof} 
It is a consequence of Lemma~\ref{Lem:m>2} and Corollary~\ref{maincor}.
\end{proof}

\subsection{The case \texorpdfstring{$1<m<2$}{1<m<2}}\label{5.4}

We cannot apply Corollary~\ref{maincor} and we have to rely directly on Proposition~\ref{Prop:K}. We recall that $m=q/p\in\R$. We shall say that the parameters $(a,m)\in\R^2$ are \emph{admissible} if and only if
\[
a=\frac1p\in\(0,\tfrac12\),\quad m=\frac qp\in(1,2)\,.
\]
With $K$ defined as in Section~\ref{section4} by~\eqref{fK}, let us start by computing $K'$. As we shall see below, the expression of $K'$ involves integer powers of $y$ and $y^{m-1}$, and can be written as a polynomial in $y$ and $z=y^{m-1}$ even if $m$ is not a rational number. This polynomial structure is useful in several computations.
\begin{lemma}\label{Lem:f} If $(a,m)$ are admissible, then $-\,K'(y)$ has the same sign as $p^2\,y^2\,f(a,m,y,z)$ where $z=y^{m-1}$ and
\begin{align*}
f(a,m,y,z)=&\;-\,3\,m\,y\,(z-1)^2\,(m\,z-1)\\
&\hspace*{12pt}+2\,(m-1-m\,y+y\,z)\(2+(1-6\,m+m^2)\,z+2\,m^2\,z^2\)\\
&\;+a\,\Big(3\,m\,y\,(z-1)^3-6\,(z-1)\,(m\,z-1)\,(m-1-m\,y+y\,z)\Big)\\
&\;+a^2\,\Big(2\,(z-1)^2\,(m-1-m\,y+y\,z)\Big)\,.
\end{align*}
\end{lemma}
\begin{proof} We set $y=x^p$ so that $x=y^{1/p}$ and $\frac{dx}{dy}=\frac1p\,y^{-1/p'}$. Let
\[
\Phi(x):=W(y)=x^{mp}-m\,x^p+m-1\quad\forall\,x\in\big[0,\gamma_m^{1/p}\big]\,,
\]
where $W$ and $h$ are as in Section~\ref{Section:notations}. Using~\eqref{D2}, we obtain
\[
\big|\sqrt q\,h'(y)\big|^2=\left|\frac{\Phi'(x)}{2\,p\,y^{1/p'}\,\sqrt{\Phi(x)}}\right|^2_{|x=y^{1/p}}\,,
\]
that is, $4\,m\,p^3\,\big|y^{1/p'}\,h'(y)\big|^2=(\Phi'(x))^2/\Phi(x)$. Hence $K$ defined by~\eqref{fK} can be rewritten as
\[
K(y)=-\,\frac12\,y^{1/p'}\,\frac d{dy}\(\frac1{y^{2/p'}\,h'(y)^2}\)=-\,2\,m\,p^3\,\frac d{dx}\(\frac{\Phi(x)}{|\Phi'(x)|^2}\).
\]
A detailed but elementary computation of the derivatives shows that $-\,K'$ has the same sign~as
\[
\frac{x^4}{q^2}\,|\Phi'(x)|^4\,\frac{d^2}{dx^2}\(\frac{\Phi(x)}{|\Phi'(x)|^2}\)=\frac{x^4}{q^2}\(6\,\Phi\,|\Phi''|^2-2\,\Phi\,\Phi'\,\Phi'''-3\,|\Phi'|^2\,\Phi''\)=p^2\,y^2\,f(a,m,y,z)\,,
\]
ending the proof of the lemma.
\end{proof}
\begin{lemma}\label{Lem:K'}
With $\Pot$ given by~\eqref{initpot} and $2<p<q<2\,p$, $K$ defined by~\eqref{fK} is monotonically decreasing.
\end{lemma}
\begin{proof} Keeping the notations of Lemma~\ref{Lem:f}, our goal is to prove that $y\mapsto f\(a,m,y,y^{m-1}\)$ is nonnegative for any $y\in(0,\gamma_m)$ whenever the parameters $(a,m)$ are \emph{admissible}.

Let us start by considering the value of $f(a,m,y,z)$ at some remarkable points.\\
$\bullet$ At $(y,z)=(0,0)$, we have $f(a,m,0,0)=2\,(1-a)\,(2-a)\,(m-1)>0$.\\
$\bullet$ At $(y,z)=(1,1)$, we have $f(a,m,1,1)=0$ but a Taylor expansion shows that
\be{f:lim}
f\(a,m,y,y^{m-1}\)=\frac1{12}\,(m-1)^3\,c_{m,a}\,(y-1)^4+O\((y-1)^5\)\quad\mbox{as}\quad y\to1
\ee
for any $a\in(0,1/2)$, where
\[
c_{m,a}=12\,m\,a\,(a-m-1)+m\(2\,m^2+7\,m+2\)\ge c_{m,1/2}=m\,(m+1)\,(2\,m-1)>0\,.
\]
This proves that $y\mapsto f\(a,m,y,y^{m-1}\)$ is positive for any $y\in(1-\varepsilon,1)\cup(1,1+\varepsilon)$ for some $\varepsilon=\varepsilon(a,m)>0$ whenever the parameters $(a,m)$ are \emph{admissible}.\\
$\bullet$ At $(y,z)=(\gamma_m,m)$, we have
\[
f(a,m,\gamma_m,m)=(m-1)^3\,c_{m,a}\,,\quad c_{m,a}=2\,a^2-3\,a\,(2\,m+2-m\,\gamma_m)+2\(2\,m^2+5\,m+2\).
\]
Using
\[
\inf_{m\in(1,2)}\frac34\,(2\,m+2-m\,\gamma_m)=\lim_{m\to1_+}\frac34\,(2\,m+2-m\,\gamma_m)=3\,(1-e/4)>1/2\,,
\]
we have
\begin{multline*}
c_{m,a}>c_{m,1/2}=(4-3\,\gamma_m)\,m^2+\(7-\tfrac32\,\gamma_m\)m+\tfrac32\\
>\lim_{m\to1_+}\((4-3\,\gamma_m)\,m^2+\(7-\tfrac32\,\gamma_m\)m+\tfrac32\)=\tfrac12\,(25-9\,e)>0\,.
\end{multline*}

\medskip In the limit as $m\to2$, we have $y=z$ and
\be{f2}
f(a,2,y,z)=2\,(1-a)\,(2-a)\,(y-1)^4\,.
\ee
Hence $f\(a,2,y,y^{m-1}\)$ is positive unless $y=1$. We are now going to take $a\in(0,1/2)$ and consider $m\in(1,2)$ as a parameter. Let us prove that for some $m_*\in(1,2)$, we have $f\(a,m,y,y^{m-1}\)\ge0$ for any $(m,y)$ such that \hbox{$m_*<m<2$} and $0\le y\le\gamma_m$. We assume by contradiction that there are two sequences $(m_k)_{k\in\N}$ and $(y_k)_{k\in\N}$ such that $1<m_k<2$ for any $k\in\N$, $\lim_{k\to+\infty}m_k=2$, $0\le y_k\le\gamma_{m_k}$ and $f\big(a,m_k,y_k,y_k^{m_k-1}\big)<0$ for any $k\in\N$. Up to the extraction of a subsequence, $(y_k)_{k\in\N}$ converges to some limit $y_\infty\in[0,2]$ and by continuity of $f$ we know that $f\(a,2,y_\infty,y_\infty\)\le0$: the only possibility is $y_\infty=1$ by~\eqref{f2}. Since $f\big(a,m_k,y_k,y_k^{m_k-1}\big)<0=f(a,m_k,1,1)$, we know that $y_k\neq1$. Since \hbox{$\lim_{k\to+\infty}y_k=1$}, this contradicts~\eqref{f:lim} or, to be precise, $|y_k-1|\ge\varepsilon(a,m_k)$, as the reader is invited to check that $\liminf_{k\to+\infty}\varepsilon(a,m_k)>0$ because $f$ is a smooth function of all of its arguments. Let
\[
m_*(a):=\inf\Big\{m\in(1,2)\,:\,f\(a,m,y,y^{m-1}\)\ge0\;\forall\,y\in[0,\gamma_m]\Big\}\,.
\]
For any $a\in(0,1/2)$, we have $m_*(a)<2$. 

We want to prove that $m_*(a)=1$. Again, let us argue by contradiction: if $m_*(a)>1$, and assume that there are two sequences $(m_k)_{k\in\N}$ and $(y_k)_{k\in\N}$ such that $1<m_k<m_*(a)$ for any $k\in\N$, $\lim_{k\to+\infty}m_k=m_*(a)$, $0\le y_k\le\gamma_{m_k}$ and $f\big(a,m_k,y_k,y_k^{m_k-1}\big)<0$ for any $k\in\N$. Up to the extraction of a subsequence, $(y_k)_{k\in\N}$ converges to some limit $y_\infty\in[0,2]$ and by continuity of $f$ we know that $f\(a,m_*(a),y_\infty,y_\infty^{m-1}\)\le0$. For the same reasons as above, $y_\infty=0$, $y_\infty=1$ and $y_\infty=\gamma_{m_*(a)}$ are excluded. Altogether, we have proved that for
\[
m=m_*(a)\,,
\]
we have $f\(a,m,y_\infty,y_\infty^{m-1}\)=0$ for some $y_\infty\in(0,1)\cup(1,\gamma_m)$ and
\[
f\(a,m,y,y^{m-1}\)\ge0\quad\forall\,y\in(0,1)\cup(1,\gamma_m)\,,
\]
so that $y_\infty$ is a local minimizer of $y\mapsto f\(a,m,y,y^{m-1}\)$. As a consequence, we have shown that for $m=m_*(a)>1$ and $y=y_\infty\neq1$, we have
\be{Contradiction}
f\(a,m,y,y^{m-1}\)=0\quad\mbox{and}\quad\frac\partial{\partial y}f\(a,m,y,y^{m-1}\)=0\,.
\ee
As we shall see next, this contradicts Lemma~\ref{Lem:finalContradiction}. Hence $y\mapsto f(a,m,y,y^{m-1})$ takes nonnegative values for any \emph{admissible} parameters $(a,m)$ for $1<m<2$. By Lemma~\ref{Lem:f}, $K'(y)\le0$, thus completing the proof.
\end{proof}

We still have to prove that~\eqref{Contradiction} has no solution $y\in(0,1)\cup(1,\gamma_m)$. Since
\[
y\,\frac\partial{\partial y}f\(a,m,y,y^{m-1}\)=y\,\frac{\partial f}{\partial y}(a,m,y,z)+(m-1)\,z\,\frac{\partial f}{\partial z}(a,m,y,z)\,,
\]
we can relax the condition $z=y^{m-1}$ and prove the slightly more general result.
\begin{lemma}\label{Lem:finalContradiction}
With the notations of Lemma~\ref{Lem:f}, assume that $m>1$, $y\in(0,\gamma_m]$ and $z\in(0,m]$. For any \emph{admissible} parameters $(a,m)$, if
\begin{subequations}
\begin{align}
&f(a,m,y,z)=0\,,\label{Elim1}\\
&y\,\frac{\partial f}{\partial y}(a,m,y,z)+(m-1)\,z\,\frac{\partial f}{\partial z}(a,m,y,z)=0\,.\label{Elim2}
\end{align}
\end{subequations}
then $z=1$.
\end{lemma}
\begin{proof} Solving the system~\eqref{Elim1}--\eqref{Elim2} is an \emph{elimination problem} because the function $f$, as defined in Lemma~\ref{Lem:f}, is a polynomial in the variables $a$, $y$ and~$z$. Since~\eqref{Elim1} is a first order equation in $y$, we can eliminate this variable and find that
\[
y=\frac{\mathsf n(a,m,z)}{\mathsf d(a,m,z)}
\]
with
\begin{align*}
\mathsf n(a,m,z):=&\;2\,(m-1)\(a^2\,(z-1)^2-3\,a\,(z-1)\,(m\,z-1)\right.\\
&\hspace*{1cm}\left.+\,2\,m^2\,z^2+\(m^2-6\,m+1\)z+2\),\\
\mathsf d(a,m,z):=&\;m\(2\,a^2\,(z-1)^2+3\,a\,(z+1)^2\,(z-1)+9\,z^2+8\,z+1\)\\
&\hspace*{1cm}-2\,z\(a^2\,(z-1)^2+3\,a\,(z-1)+z+2\)\\
&\hspace*{1cm}-m^2\,z\(6\,a\,(z-1)+z^2+8\,z+9\)+2\,m^3\,z\,(2\,z+1)\,.
\end{align*}
After replacing, we obtain after lengthy but elementary computations that solving~\eqref{Elim2} under the condition $z\neq1$ is reduced to a second order equation in $z$, whose discriminant~is
\[
\delta(a,m):=-\,3\,(a-1)^2\,(m-1)^2\,(a-m)^2\(5\,a^2\!-\!10\,a\,(m+1)\!-\!3\,m^2\!+\!14\,m\!-\!3\).
\]
Since $5\,a^2-10\,a\,(m+1)-3\,m^2+14\,m-3$ takes only positive values for \emph{admissible} $(a,m)$, there are no other roots than $z=1$. This is the desired contradiction, which completes the proof.
\end{proof}
Then, as a straightforward consequence of Lemmas~\ref{Prop:K} and~\ref{Lem:K'}, we have the
\begin{lemma}\label{Lem:m>222} If $1<m<2$, then the minimal period $T(E)$ is a monotonically increasing function of~$E\in(0,E_*)$.\end{lemma}

Finally, we can conclude with the proof of our main theorem.
\begin{proof}[of Theorem~\ref{Thm:Main3}] The monotonicity of the minimal period $T$ of the solution of~\eqref{ELS2} follows from Lemmas~\ref{Lem:m=2},~\ref{Lem:m>22}, and~\ref{Lem:m>222}. The asymptotic behaviours of $T$ as $E\to0_+$ and as $E\to( E_*)_-$ are established in Lemma~\ref{asymp}.
\end{proof}

\appendix\section{A variational problem}\label{App}

A central motivation for studying~\eqref{ELS2} arises from the minimization problem
\be{functional}
\mu(\lambda):=\inf_{f\in\mathrm W^{1,p}(\mathbb S^1)\setminus{\{0\}}}\frac{\nrm{f'}p^2+\lambda\,\nrm fp^2}{\nrm fq^2}
\ee
where $q>p$ is an arbitrary exponent and $\mathbb S^1$ is the unit circle. This problem can also be seen as the search for the optimal constant in the \emph{interpolation inequality}
\[
\nrm{f'}p^2+\lambda\,\nrm fp^2\ge\mu(\lambda)\,\nrm fq^2\quad\forall\,f\in\mathrm W^{1,p}(\mathbb S^1)\,.
\]
Testing the inequality by constant functions shows that $\mu(\lambda)\le\bar\mu(\lambda):=\lambda\,|\mathbb S^1|^{\frac2p-\frac2q}$. If $p=2$, one can consider the interpolation inequality on $\S^d$ for any integer $d\ge1$ and it is well known from the \emph{carr\'e du champ} method~\cite{Bakry1985,MR3155209} that equality holds if and only if $\lambda\le d/(q-2)$. If $\lambda>d/(q-2)$, we have $\mu(\lambda)<\bar\mu(\lambda)$ and optimal functions are non-constant, so that \emph{symmetry breaking} occurs. This is a basic mechanism in \emph{phase transition} theory that can be interpreted as a bifurcation of a branch of non-trivial functions from a branch of constant functions. An important question is therefore to find the largest value of $\lambda>0$ such that $\mu(\lambda)=\bar\mu(\lambda)$. This is an open question for $p>2$, even if $d=1$.

In dimension $d=1$, the bifurcation problem degenerates in the limit case $p=2$, for which $\lambda_1=\lambda_2=1/(q-2)$ according to~\cite{Bakry1985}. We refer to~\cite[Section~1]{BDL} for an introduction to the minimization problem~\eqref{functional} for $p=2$, the issue of the branches and the monotonicity of the period problem. In that case, proving that \emph{symmetry breaking} occurs if and only if $\lambda>1/(q-2)$ can be reduced to a proof of the \emph{monotonicity of the minimal period} using Chicone's criterion~\cite[Theorem~A]{CH}, which provides an interesting alternative to the \emph{carr\'e du champ} method. The study of bifurcation problems using the period function goes back to~\cite{MR607786} in case of equations with cubic non-linearities and was later extended to various classes of Hamiltonian systems in~\cite{Schaaf,Rothe_1993,Cop,Coll,Gas}. It is therefore natural to consider the case $p>2$.

The minimization problem for $p>2$ was studied in~\cite{DolGAMA-2020}. There is an optimal function for~\eqref{functional} and the corresponding Euler-Lagrange equation turns out to be the nonlinear differential equation with nonlocal terms given by
\be{ELS0}
-\,\nrm{f'}p^{2-p}\,\big(\phi_p(f')\big)'+\lambda\,\nrm fp^{2-p}\,\phi_p(f)=\mu(\lambda)\,\nrm fq^{2-q}\,\phi_q(f)\,,
\ee
where we look for positive solutions on $W^{1,p}(\mathbb S^1)\setminus{\{0\}}$ or equivalently positive $2\pi$-periodic solutions on $\R$. Using the homogeneity and the scaling properties of~\eqref{ELS0}, one can get rid of the norms, $\lambda$ and $\mu$ so that the problem is reduced to~\eqref{ELS2}, with a period which is not anymore $2\,\pi$. The study of the branches of positive solutions of~\eqref{ELS0} parametrized by $\lambda>0$ and the properties of the function $\lambda\mapsto\mu(\lambda)$ is reduced to the study of all positive periodic solutions on~$\R$ of~\eqref{ELS2} and the corresponding minimal periods. Details can be found in~\cite[Section~4]{DolGAMA-2020}.

So far, we do not know the precise value of $\lambda$ for which there is symmetry breaking but according to~\cite{DolGAMA-2020} \emph{rigidity} holds if $0<\lambda<\lambda_1$ for some $\lambda_1>0$, where rigidity means that any positive solution of~\eqref{ELS0} is a constant. In that range, we have $\mu(\lambda)=\bar\mu(\lambda)$. On the contrary, one can prove that \emph{symmetry breaking} occurs if $\lambda>\lambda_2$ for some $\lambda_2>\lambda_1$, so that $\mu(\lambda)<\bar\mu(\lambda)$ and~\eqref{ELS0} admits non-constant positive solutions for any $\lambda>\lambda_2$. If $p=2$, a precise description of the threshold value of $\lambda$ is known in the framework of Markov processes if $q$ is not too large (see~\cite{MR3155209} for an overview with historical references that go back to~\cite{Bakry1985}) and from~\cite{MR1134481,Demange_2008,DEKL,DolEsLa-APDE2014,Dolbeault20141338,1504,Dolbeault_2020} using \emph{entropy methods} applied to nonlinear elliptic and parabolic equations; also see~\cite{Dolbeault:2021wb} for an overview and extensions to various related variational problems. The results of~\cite{DolGAMA-2020} are also based on \emph{entropy methods}. Almost nothing is known beyond~\cite{DolGAMA-2020} if $p>2$, even for $d=1$. The results of this paper are a contribution to a better understanding of the fundamental properties of the solutions of~\eqref{ELS2} in the simplest case, for $p>2$.

\vspace*{-0.5cm}\section{A comparison of the main results}\label{App:Comparison}

Here we consider the results of Theorems~\ref{Thm:Main1},~\ref{Thm:Main2} and~\ref{Thm:Main3} when the potential $\Pot$ is given by~\eqref{initpot} with $2<p<q$. Let $m_*(p)$ be the unique positive root of $m\mapsto(p-1)\,m^{m/(m-1)}-p\,m-p+1$,
\[
p_*:=\frac{e-1}{e-2}\approx2.39221\,,
\]
and recall that $m=q/p$, $B=m^{1/(q-p)}$, $\mathcal R(w):=|\Pot'(w)|^2-p'\,\Pot(w)\,\Pot''(w)$ with \hbox{$p'=p/(p-1)$}.
\begin{proposition}\label{Prop:Thms1-3} With the above notation, we have $m_*(p)>1$ if and only if $p>p_*$. In the range $p_*<p<q<p\,m_*(p)$, the minimum of $\mathcal R$ on $[0,B]$ is negative.\end{proposition}
\noindent As a consequence, Theorem~\ref{Thm:Main1} applied to $\Pot(w)$ given by~\eqref{initpot} does not cover the whole range of exponents $q$ in Theorem~\ref{Thm:Main3}.
\begin{proof}[of Proposition~\ref{Prop:Thms1-3}] A direct computation of 
\[
\mathcal R(B)=\frac1{p-1}\,(m-1)^2\,m^{\frac{p-2}{p\,(m-1)}-1}\((p-1)\,m^\frac m{m-1}-p\,m-p+1\)
\]
shows that $\mathcal R(B)<0$ if $p>p_*$ and $p<q<p\,m_*(p)$.\end{proof}

\medskip It is a natural question to ask whether the result of Theorem~\ref{Thm:Main3} is a consequence of Theorem~\ref{Thm:Main2}, or not. Although there are numerical evidences, at least for some values of $p$ and $q$, we are not able to prove that $\Pot/(\Pot')^2$ is convex if $\Pot$ is given by~\eqref{initpot} with $2<p<q$, and we do not know the answer. The proofs of Theorem~\ref{Thm:Main2} and Theorem~\ref{Thm:Main3} share some computations, but also differ significantly, and this deserves further comments.

In Theorem~\ref{Thm:Main3}, the first ingredient is the change of variables $w\mapsto y=w^p$, which allows to define $h(y)$ by~\eqref{D1}, such that $|h(y)|=\sqrt{\Pot(w(y))}$ with $w(y)=y^{1/p}$, and the key quantity $K(y)$ by~\eqref{fK}. With $w=w(y)$, an elementary computation shows that
\[
K(y)=-\,2\,p\,y^{2/p'}\,\frac{|\Pot'(w)|^2-2\,\Pot(w)\,\Pot''(w)}{|\Pot'(w)|^3}\,.
\]
Proposition~\ref{Prop:K} and Corollary~\ref{maincor} are central in the proof of Theorem~\ref{Thm:Main3} and both rely on the assumption that $K$ is decreasing on $[0,B^p]$. In Theorem~\ref{Thm:Main2}, the key assumption is the condition that $\Pot/(\Pot')^2$ is a convex function, which  simply means that 
\[
w\mapsto-\,\frac{|\Pot'(w)|^2-2\,\Pot(w)\,\Pot''(w)}{\Pot'(w)^3}
\]
is nonincreasing on $[0,B]$. This is similar to the condition on $K$, but only if $\Pot'(w)$ is positive.

Compared to the proof of Theorem~\ref{Thm:Main2}, Theorem~\ref{Thm:Main3} involves several additional ingredients. The proof of Proposition~\ref{Prop:K} is based on the change of variables $y\mapsto\theta\in(-\pi/2,\pi/2)$ such that $\sqrt E\,\sin\theta=h(y)$ and allows us to write an identity in terms of $K\big(y(E,\theta)\big)-K\big(y(E,-\,\theta)\big)$ with $\theta\in(0,\pi/2)$. This difference has a sign precisely because of the monotonicity of $K$. Finally, the computation of $K'$ involves integer powers of $y$ and $z=y^{m-1}$, considered as independent variables and the proof that that $K'<0$ relies on elimination techniques applied to the polynomial in $y$ and $z$. It is not excluded that a similar technique could be used to prove directly that $\Pot/(\Pot')^2$ is convex if $\Pot$ is given by~\eqref{initpot}, but we have not been able to implement it.

\vspace*{-0.8cm}

\hfuzz=4pt
\affiliationone{J.~Dolbeault\\ 
CEREMADE (CNRS UMR n$^\circ$ 7534),\\ PSL Univ., Universit\'e Paris-Dauphine,\\ Place de Lattre de Tassigny,\\ 75775 Paris 16, France
\email{dolbeaul@ceremade.dauphine.fr}}
\affiliationtwo{M.~Garc\'ia-Huidobro\\ 
Departamento de Matem\'{a}ticas,\\ Pontificia Universidad Cat\'{o}lica de Chile,\\ Casilla 306, Correo 22,\\ Santiago, Chile
\email{mgarcica@uc.cl}}
\affiliationthree{R.~Man\'asevich\\ 
Departamento de Ingenier\'ia Matem\'atica and CMM (CNRS IRL n$^\circ$ 2807),\\
FCFM, Universidad de Chile,\\ Casilla 170, Correo 3,\\Santiago, Chile
\email{manasevi@dim.uchile.cl}}

\end{document}